%% file: Top._models_of_arithmetic__arXiv-ver3_.tex
\documentclass[10pt]{amsart}

\title[Topological models of arithmetic]{Topological models of arithmetic}

\author{Ali Enayat}
  \address[Ali Enayat]{Department of Philosophy, Linguistics, and Theory of Science, University of Gothenburg, Box 200, SE 405 30, Gothenburg, Sweden}
  \email{ali.enayat@gu.se}
  \urladdr{https://www.gu.se/en/about/find-staff/alienayat}

\author{Joel David Hamkins}
 \address[Joel David Hamkins]{Professor of Logic, University of Oxford, and
           Sir Peter Strawson Fellow in Philosophy, University College, Oxford OX1 4BH}
 \email{joel.hamkins@philosophy.ox.ac.uk}
 \urladdr{http://jdh.hamkins.org}

\author{Bartosz Wcis{\l}o}
 \address[Bartosz Wcis{\l}o]{Institute of Mathematics, Polish Academy of Sciences}
 \email{b.wcislo@impan.pl}
 \urladdr{http://duch.mimuw.edu.pl/~bw276697/english/main.html}

\thanks{Commentary can be made about this article on the second author's blog at \href{http://jdh.hamkins.org/topological-models-of-arithmetic}{http://jdh.hamkins.org/topological-models-of-arithmetic}.}
\usepackage{amsthm} 
\usepackage{latexsym,amsfonts,amsmath,amssymb,mathrsfs}
\usepackage[hidelinks]{hyperref}
\usepackage{relsize}
\usepackage[rgb,dvipsnames]{xcolor}
\usepackage{tikz}
\usetikzlibrary{arrows,arrows.meta,petri,topaths,positioning,shapes,shapes.misc,patterns,calc,decorations.pathreplacing,hobby}
\usepackage[shortlabels]{enumitem} 

\include{MathMacrosJDH}
\usepackage[backend=bibtex,style=alphabetic,maxbibnames=15,maxcitenames=6,dateabbrev=false]{biblatex}
\renewcommand{\UrlFont}{\sffamily\smaller} 
\renewbibmacro{in:}{\ifentrytype{article}{}{\printtext{\bibstring{in}\intitlepunct}}} 
\DeclareFieldFormat{url}{{\relsize{-2}\UrlFont\url{#1}}} 
\DeclareFieldFormat{urldate}{
  (version \thefield{urlday}\addspace%
  \mkbibmonth{\thefield{urlmonth}}\addspace%
  \thefield{urlyear}\isdot)}
\DeclareFieldFormat{eprint:arxiv}{
  \ifhyperref
    {\href{http://arxiv.org/abs/#1}{%
        arXiv\addcolon\nolinkurl{#1}}\iffieldundef{eprintclass}{}{\UrlFont{\mkbibbrackets{\thefield{eprintclass}}}}}
    {arXiv\addcolon\nolinkurl{#1}\iffieldundef{eprintclass}{}{\UrlFont{\mkbibbrackets{\thefield{eprintclass}}}}}}
\bibliography{HamkinsBiblio,MathBiblio,WebPosts}

\usepackage{biblatex}
\usepackage[utf8]{inputenc}
\usepackage[english]{babel}

\calclayout

\begin{document}

\begin{abstract}
 Ali Enayat had asked whether there is a model of PA (Peano Arithmetic) that can be represented as  $\<\Q,\oplus,\otimes>$, where $\oplus$ and $\otimes$ are continuous functions on the rationals $\Q$. We prove, affirmatively, that indeed every countable model of \PA\ has such a continuous presentation on the rationals. More generally, we investigate the topological spaces that arise as such topological models of arithmetic. Finite dimensional Euclidean spaces $\R^n$, and compact Hausdorff spaces do not, and neither does any Suslin line; many other spaces do. The status of the space of irrationals remains open.
\end{abstract}

\maketitle

\section{Introduction}\label{Section.Introduction}

This paper arose from the following question that asks whether a model of arithmetic could be continuously presented on the rational numbers (the motivation for this question is explained in Remark \ref{Motivation for Main Question}).

\begin{mainquestion}[Enayat~\cite{Enayat2009:Borel-structures-via-models-of-arithmetic}]
 Are there continuous functions $\oplus$ and $\otimes$ on the rational numbers $\Q$, such that $\<\Q,\oplus,\otimes>$ is a model of \PA\ ?
\end{mainquestion}

\PA\ (Peano Arithmetic) consists of the first order theory $\PA^-$ of the non-negative parts of discretely ordered rings, plus the induction principle for assertions in the first order language of arithmetic. The natural numbers $\<\N,+,\cdot>$ form what is known as the \emph{standard} model of \PA, but there are also many nonstandard models, including continuum many non-isomorphic countable models. Although we take \PA\ as a central case, most of our arguments do not use the full strength of \PA, and our analysis applies to models of various weaker arithmetic theories. Consequently throughout the paper, we use the expression ``a model of arithmetic'' to refer to a model of a sufficiently strong fragment of \PA. The fragments that appear in the results of this paper are: Successor Arithmetic (the first order theory of $\<\mathbb{N}, 0,S>$, where $S(x)=x+1$), Presburger Arithmetic (the first order theory of $\<\mathbb{N},+>$), $\PA^-$, $\mathrm{IOpen}$ ($\PA^-$ plus induction axioms for quantifier free formulas),  $\text{\rm I}\Delta_0$ ($\PA^-$ plus induction axioms for bounded formulas), $\text{\rm I}\Delta_0+\text{Exp}$ (where $\text{Exp}$ asserts the totality of the exponential function), and $\text{\rm I}\Sigma_1$ ($\PA^-$ plus induction axioms for existential ($\Sigma_1$) formulas).   We refer the reader to \cite{HajekPudlak1998:Metamathematics-of-first-order-arithmetic} and \cite{Kaye1991:ModelsOfPeanoArithemtic} for general background on arithmetical theories and their model theory.

We shall answer the main question affirmatively, and indeed, our main theorem (theorem~\ref{maintheorem}) shows that every countable model of $\text{\rm I}\Delta_0$ is continuously presented on $\Q$. We define generally that a \emph{topological} model of arithmetic is a topological space $X$ equipped with continuous functions $\oplus$ and $\otimes$, for which $\<X,\oplus,\otimes>$ satisfies the desired arithmetic theory. In such a case, we shall say that the underlying space $X$ continuously \emph{supports} a model of arithmetic and that the model is continuously \emph{presented} on/upon the space $X$.

\begin{question}
 Which topological spaces support a topological model of arithmetic?
\end{question}

In section~\ref{Section.Non-examples}, we shall prove that many familiar topological spaces do not support models of arithmetic. The list of such spaces includes every finite dimensional Euclidean space $\R^n$, the long line, the Cantor space, and Suslin lines. Meanwhile, there are many other spaces that do support topological models, including many uncountable subspaces of the plane $\R^2$. It remains an open question whether any uncountable Polish space, including the Baire space, can support a topological model of arithmetic.

In light of our answer to the main question in the case of $\Q$, where we showed not only that some models of \PA\ are continuously presented on $\Q$, but that all countable models of \PA\ can be continuously presented on $\Q$---and the same is true of the indiscrete and discrete spaces---it seems natural to ask whether this none/all phenomenon continues with other spaces.

\begin{question}\label{Question.If-one-then-all?}
 If a topological space $X$ continuously supports a model of \PA, then does $X$ continuously support a copy of every model of \PA\ of that cardinality?
\end{question}

We answer question \ref{Question.If-one-then-all?} negatively in theorem~\ref{Theorem.Some-not-all-non-kappa-like}, showing that in every infinite cardinality, there are spaces that continuously support some, but not all models of arithmetic in that cardinality. Thus, the underlying topology is able to distinguish the arithmetic structure. One might similarly ask whether it can distinguish the arithmetic \emph{theory} of the structures.

\begin{question}\label{Question.Topology-cannot-distinguish-the-theory}
 If a topological space $X$ continuously supports a model of \PA, then does it continuously support a model of every completion of \PA?
\end{question}

We also answer question \ref{Question.Topology-cannot-distinguish-the-theory} negatively in corollary \ref{Corollary.X-with-only-standard-model-presented}, although the question remains open if one insists that the topology is Hausdorff or that the space is uncountable.

\section{Every countable model is continuously presented on the rationals}

In this section we prove the main theorem, part (a) of which answers the main question. A soft argument can be used to show that part (b) of the main theorem implies part (a) of the theorem (more specifically, this follows from the line of reasoning used to prove lemma \ref{Lemma.Upward-homogeneity}). However, for clarity of exposition we found it preferable to establish (a) in full detail, and then explain the fine tuning that is needed to establish (b).  Let us also point out that later in theorem \ref{Theorem.Countable-metric-space-presentations}, we show that each countable model of arithmetic also lends itself to continuous presentations on metric spaces obtained by augmenting the rationals $\Q$ with a prescribed finite or countable number of isolated points. In contrast, as noted in remark \ref{Remark on ordered rings}, if the ring $\Z^M$ of a model $M$ of arithmetic admits a continuous presentation on a countable metric space $X$, then either $X$ is homeomorphic to the discrete space $\N$, or $X$ is homeomorphic to the rationals $\Q$.

\begin{maintheorem}\label{maintheorem} Let $M$ be a countable model of the fragment $\text{\rm I}\Delta_0$ of \PA.\

\begin{enumerate}

\item [\rm{(a)}] $M$ admits a continuous presentation on the rationals $\Q$.

\item [\rm{(b)}] The ring of integers $\Z^M$ of $M$ admits a continuous presentation on the rationals $\Q$.
 \end{enumerate}
\end{maintheorem}

\newcommand\fdlt{\mathrel{\triangleleft}}
\begin{proof} (a). We shall prove (a) first for the standard model of arithmetic $\<\N,+,\cdot>$. Every school child knows that when computing integer sums and products by the usual algorithms, the final digits of the result $x+y$ or $x\cdot y$ are completely determined by the corresponding final digits of the inputs $x$ and $y$. Presented with only final segments of the input, the child can nevertheless proceed to compute the corresponding final segments of the sum or product.
\begin{equation*}\small\begin{array}{rcr}
\cdots1261\quad & \qquad                  & \cdots1261\quad\\
\underline{+\quad\cdots 153\quad}&\qquad & \underline{\times\quad\cdots 153\quad}\\
\cdots414\quad & \qquad                  &     \cdots3783\quad\\
               &                         &    \cdots6305\phantom{3}\quad\\
               &                         &  \cdots1261\phantom{53}\quad\\
               &                         &  \underline{\quad\cdots\cdots\phantom{253}\quad}\\
               &                         &  \cdots933\quad\\
\end{array}\end{equation*}

This phenomenon amounts exactly to the continuity of addition and multiplication with respect to what we call the \emph{final-digits} topology on $\N$, which is the topology having basic open sets $U_s$, the set of numbers whose binary representations end with the digits $s$, for any finite binary string $s$. (A similar idea can be used with any base, although the topologies can differ.) In the $U_s$ notation, we include the number that would arise by deleting initial $0$s from $s$; for example, $6\in U_{00110}$. Addition and multiplication are continuous in this topology, because if $x+y$ or $x\cdot y$ has final digits $s$, then by the school-child's observation, this is ensured by corresponding final digits in $x$ and $y$, and so $(x,y)$ has an open neighborhood in the final-digits product space, whose image under the sum or product, respectively, is contained in $U_s$.

Let us make several elementary observations about the topology. The sets $U_s$ do indeed form the basis of a topology, because $U_s\intersect U_t$ is empty, if $s$ and $t$ disagree on some digit (comparing from the right), or else it is either $U_s$ or $U_t$, depending on which sequence is longer. The topology is Hausdorff, because different numbers are distinguished by sufficiently long segments of final digits. There are no isolated points, because every basic open set $U_s$ has infinitely many elements. Every basic open set $U_s$ is clopen, since the complement of $U_s$ is the union of $U_t$, where $t$ conflicts on some digit with $s$. The topology is actually the same as the metric topology generated by the $2$-adic valuation, which assigns the distance between two numbers as $2^{-k}$, when $k$ is largest such that $2^k$ divides their difference; the set $U_s$ is an open ball in this metric, centered at the number represented by $s$ and of radius $2^{-|s|}$. 
(One can also see that it is metric by the Urysohn metrization theorem \cite[theorem~34.1]{Munkres2000:Topology}, since it is a Hausdorff space with a countable clopen basis, and therefore regular.) By a theorem of \Sierpinski\ \cite{Sierpinski1920:Sur-une-propriete-topolologique-des-ensembles-denombrables, Francis2012:Two-topological-uniqueness-theorems}, every countable metric space without isolated points is homeomorphic to the rational line $\Q$, and so we conclude that the final-digits topology on $\N$ is homeomorphic to $\Q$. We've therefore proved that the standard model of arithmetic $\N$ has a continuous presentation on $\Q$, as desired.

But let us belabor the argument somewhat, since we find it noteworthy that the final-digits topology (or equivalently, the $2$-adic metric topology on $\N$) is precisely the order topology of a certain definable order on $\N$, what we call the \emph{final-digits} order, an endless dense linear order, which is therefore order-isomorphic and thus also homeomorphic to the rational line $\Q$, as desired. This will enable the generalization of corollary~\ref{Corollary.Continuous-on-Q^M}.
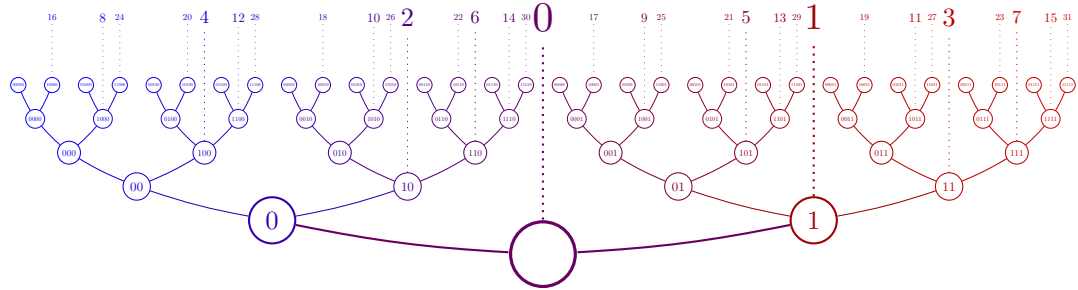
\begin{figure}[h]
\begin{tikzpicture}[scale=.45,xscale=1,>=Stealth,treenode/.style={circle,draw},toparrow/.style={dotted}]
\colorlet{numcolor}{black!35!red!62!blue}
\draw[numcolor] (16,0) node[treenode,scale=1.4,very thick] (root) {\phantom{$0$}};
\draw[numcolor] (root)+(0,7) node[scale=1.6,thick] (number0) {$0$};
\draw[toparrow,numcolor,thick] (root) -- (number0);
\foreach \i in {0,1}
  {
  \pgfmathsetmacro\num{int(\i)}
  \pgfmathsetmacro\c{4*(16*\i+8)}
  \colorlet{numcolor\i}{black!35!red!\c!blue}
  \draw (16*\i+8,1) node[treenode,scale=1,numcolor\i,thick] (\i) {$\i$};
  \ifthenelse{\i=1}{\draw[numcolor,thick,bend right=4] (root) to (\i);
                    \draw[numcolor\i] (\i)+(0,6) node[scale=1.4] (number\i) {$1$};
                    \draw[toparrow,numcolor\i,thick] (\i) -- (number\i);}
                   {\draw[numcolor,thick,bend left=4] (root) to (\i);}
  \foreach \j in {0,1}
    {
    \pgfmathsetmacro\num{int(2*\j+\i)}
    \pgfmathsetmacro\c{4*(16*\i+8*\j+4)}
    \colorlet{numcolor\j\i}{black!35!red!\c!blue}
    \draw (16*\i+8*\j+4,2) node[treenode,scale=.49,numcolor\j\i] (\j\i) {$\j\i$};
    \ifthenelse{\j=1}{\draw[numcolor\i,bend right=4] (\i) to (\j\i);
                      \draw[numcolor\j\i] (\j\i)+(0,5) node[scale=1] (number\j\i) {$\num$};
                      \draw[toparrow,numcolor\j\i] (\j\i) -- (number\j\i);}
                     {\draw[numcolor\i,bend left=4] (\i) to (\j\i);}
    \foreach \k in {0,1}
      {
      \pgfmathsetmacro\num{int(4*\k+2*\j+\i)}
      \pgfmathsetmacro\c{4*(16*\i+8*\j+4*\k+2)}
      \colorlet{numcolor\k\j\i}{black!35!red!\c!blue}
      \draw (16*\i+8*\j+4*\k+2,3) node[treenode,scale=.343,numcolor\k\j\i,thin] (\k\j\i) {$\k\j\i$};
      \ifthenelse{\k=1}{\draw[numcolor\j\i,thin,bend right=4] (\j\i) to (\k\j\i);
                        \draw[numcolor\k\j\i] (\k\j\i)+(0,4) node[scale=.7] (number\k\j\i) {$\num$};
                        \draw[toparrow,numcolor\k\j\i,thin] (\k\j\i) -- (number\k\j\i);}
                       {\draw[numcolor\j\i,thin,bend left=4] (\j\i) to (\k\j\i);}
      \foreach \l in {0,1}
        {
        \pgfmathsetmacro\num{int(8*\l+4*\k+2*\j+\i)}
        \pgfmathsetmacro\c{4*(16*\i+8*\j+4*\k+2*\l+1)}
        \colorlet{numcolor\l\k\j\i}{black!35!red!\c!blue}
        \draw (16*\i+8*\j+4*\k+2*\l+1,4) node[treenode,scale=.24,numcolor\l\k\j\i,very thin]
                                                       (\l\k\j\i) {$\l\k\j\i$};
        \ifthenelse{\l=1}{\draw[numcolor\k\j\i,very thin,bend right=4] (\k\j\i) to (\l\k\j\i);
                          \draw[numcolor\l\k\j\i] (\l\k\j\i)+(0,3) node[scale=.49] (number\l\k\j\i) {$\num$};
                          \draw[toparrow,numcolor\l\k\j\i,very thin] (\l\k\j\i) -- (number\l\k\j\i);}
                         {\draw[numcolor\k\j\i,very thin,bend left=4] (\k\j\i) to (\l\k\j\i);}
        \foreach \m in {0,1}
          {
          \pgfmathsetmacro\num{int(16*\m+8*\l+4*\k+2*\j+\i)}
          \pgfmathsetmacro\c{4*(16*\i+8*\j+4*\k+2*\l+\m)}
          \colorlet{numcolor\m\l\k\j\i}{black!35!red!\c!blue}
          \draw (16*\i+8*\j+4*\k+2*\l+\m+.5,5) node[treenode,scale=.168,numcolor\m\l\k\j\i,ultra thin]
                                (\m\l\k\j\i) {$\m\l\k\j\i$};
          \ifthenelse{\m=1}{\draw[numcolor\l\k\j\i,ultra thin,bend right=4] (\l\k\j\i) to (\m\l\k\j\i);
                            \draw[numcolor\m\l\k\j\i] (\m\l\k\j\i)+(0,2) node[scale=.343] (number\m\l\k\j\i) {$\num$};
                            \draw[toparrow,numcolor\m\l\k\j\i,very thin] (\m\l\k\j\i) -- (number\m\l\k\j\i);}
                           {\draw[numcolor\l\k\j\i,ultra thin,bend left=4] (\l\k\j\i) to (\m\l\k\j\i);}
          }
        }
      }
    }
  }
\end{tikzpicture}
\caption{The final-digits order on the natural numbers}\label{Figure.Final-digits-order}
\end{figure}

Specifically, the final-digits order on the natural numbers, pictured in figure~\ref{Figure.Final-digits-order}, is the order induced from the lexical order on the finite binary representations, but considering the digits from right-to-left, giving higher priority in the lexical comparison to the low-value final digits of the number. More precisely, the final-digits order $n\fdlt m$ holds, if at the first point of disagreement (from the right) in their binary representation, $n$ has $0$ and $m$ has $1$; or if there is no disagreement, because one of them is longer, then the longer number is lower, if the next digit is $0$, and higher, if it is $1$ (this is not the same as treating missing initial digits as zero). Thus, the even numbers appear as the left half of the order, since their final digit is $0$, and the odd numbers as the right half, since their final digit is $1$, and $0$ is directly in the middle; indeed, the highly even numbers, whose representations end with a lot of zeros, appear further and further to the left, while the highly odd numbers, which end with many ones, appear further and further to the right. If one does not allow initial $0$s in the binary representation of numbers, then note that zero is represented in binary by the empty sequence. It is evident that the final-digits order is an endless dense linear order on $\N$, just as the corresponding lexical order on finite binary strings is an endless dense linear order.

The basic open set $U_s$ of numbers having final digits $s$ is an open set in this order, since any number ending with $s$ is above a number with binary form $100\cdots0s$ and below a number with binary form $11\cdots 1s$ in the final-digits order; so $U_s$ is a union of intervals in the final-digits order. Conversely, every interval in the final-digits order is open in the final-digits topology, because if $n\fdlt x\fdlt m$, then this is determined by some final segment of the digits of $x$ (appending initial $0$s if necessary), and so there is some $U_s$ containing $x$ and contained in the interval between $n$ and $m$. Thus, the final-digits topology is the precisely the same as the order topology of the final-digits order, which is a definable endless dense linear order on $\N$. Since this order is isomorphic and hence homeomorphic to the rational line $\Q$, we conclude again that $\<\N,+,\cdot>$ admits a continuous presentation on $\Q$.

The reasoning of the above paragraphs can be undertaken in $\text{\rm I}\Delta_0$. More explicitly, it is well-known that the theory $\text{\rm I}\Delta_0$ can prove that every number has a (unique) binary representation. This is based on the remarkable fact (see [HP, Section V3(c)]) that the graph of the exponential function has a $\Delta_0$ definition (provably in $\text{\rm I}\Delta_0$). Moreover, it can be readily verified that within $\text{\rm I}\Delta_0$, the finite-digits order is an endless dense linear order. It follows that $M$ can see that its addition and multiplication are continuous with respect to the order topology of its final-digits order. Since $M$ is countable, the final-digits order of $M$ makes it a countable endless dense linear order, which by Cantor's theorem is therefore order-isomorphic and hence homeomorphic to $\Q$. Thus, $M$ has a continuous presentation on the rational line $\Q$, as desired.

The executive summary of the proof is: the arithmetic of the standard model $\N$ is continuous with respect to the final-digits topology, which is the same as the $2$-adic metric topology on $\N$, and this is homeomorphic to the rational line, because it is the order topology of the final-digits order, a definable endless dense linear order; applied in a nonstandard model $M$ of a sufficiently strong fragment of \PA\, this observation means the arithmetic of $M$ is continuous with respect to its rational line $\Q^M$, which for a countable model $M$ is isomorphic to the actual rational line $\Q$, and so such an $M$ is continuously presentable upon $\Q$. This concludes the proof of part (a) of Theorem \ref{maintheorem}. \medskip

(b)
 Consider first the standard integers $\Z$. The analogue of the school-child's observation is that addition and multiplication of integers is continuous with respect to the signed-finals-digits topology, in the sense that if you know the final digits of the numbers and their sign (positive or negative), then you can know the corresponding final digits and sign of the output. If we represent integers by their binary digits with an indication for the sign---we represent zero by the empty sequence---then the basic open sets in this topology have the form $U_{s+}$ or $U_{s-}$, for any nonempty binary sequence $s$, where the first set is the set of positive numbers whose binary expansion has final digits $s$, and the second set refers analogously to negative numbers, plus the sets of negative numbers $U_{-}$, positive numbers $U_{+}$ and the sets $U_{00\cdots0}$, which are the numbers whose final digits end all in zeros, but including both positive and negative numbers. This topology is the same as the order topology of the signed-final-digits order, pictured in figure~\ref{Figure.Signed-final-digits-order} below. This is an endless dense linear order, defined by the lexical order on the signed binary representation (but note how the sign affects the direction of the lexical comparison for negative numbers), and so the topology is homeomorphic to the rational line.
\newcommand{\unaryminus}{\llap{\scalebox{0.5}[1.0]{\( - \)}}}
\begin{figure}[h]
\begin{tikzpicture}[scale=.4,xscale=.55,>=Stealth,treenode/.style={circle,draw},toparrow/.style={dotted}]
\colorlet{rootcolor}{black!35!blue}
\draw[rootcolor] (0,-1) node[treenode,scale=1.4,very thick] (root) {\phantom{$0$}};
\draw[rootcolor] (root)+(0,7) node[scale=1.6,thick] (number0) {$0$};
\draw[toparrow,rootcolor,thick] (root) -- (number0);
\foreach \s in {-1,1}
  {
  \colorlet{numcolor}{black!35!red!62!blue}
  \pgfmathsetmacro\x{16}
  \ifthenelse{\s=1}
             {\draw (\s*\x,0) node[treenode,scale=1,numcolor,thick] (sign\s) {$+$};
              \draw[rootcolor,thick,bend right=3] (root) to (sign\s);}
             {\draw (\s*\x,0) node[treenode,scale=1,numcolor,thick] (sign\s) {$-$};
              \draw[rootcolor,thick,bend left=3] (root) to (sign\s);}
  \foreach \i in {0,1}
    {
    \pgfmathsetmacro\num{int(\i)}
    \pgfmathsetmacro\x{16*\i+8}
    \pgfmathsetmacro\c{4*\x}
    \colorlet{numcolor\i}{black!35!red!\c!blue}
    \draw (\s*\x,1) node[treenode,scale=.49,numcolor\i] (\i) {$\i$};
    \ifthenelse{\i=1}{\draw[numcolor,thick,bend right=\s*3] (sign\s) to (\i);
                      \draw[numcolor\i] (\i)+(0,5) node[scale=1] (number\i) {$\ifthenelse{\s=1}{}{\unaryminus}1$};
                      \draw[toparrow,numcolor\i] (\i) -- (number\i);}
                     {\draw[numcolor,thick,bend left=\s*3] (sign\s) to (\i);}
    \foreach \j in {0,1}
      {
      \pgfmathsetmacro\num{int(2*\j+\i)}
      \pgfmathsetmacro\x{16*\i+8*\j+4}
      \pgfmathsetmacro\c{4*\x}
      \colorlet{numcolor\j\i}{black!35!red!\c!blue}
      \draw (\s*\x,2) node[treenode,scale=.343,numcolor\j\i,thin] (\j\i) {$\j\i$};
      \ifthenelse{\j=1}{\draw[numcolor\i,bend right=\s*3] (\i) to (\j\i);
                        \draw[numcolor\j\i] (\j\i)+(0,4) node[scale=.7] (number\j\i) {$\ifthenelse{\s=1}{}{\unaryminus}\num$};
                        \draw[toparrow,numcolor\j\i,thin] (\j\i) -- (number\j\i);}
                       {\draw[numcolor\i,bend left=\s*3] (\i) to (\j\i);}
      \foreach \k in {0,1}
        {
        \pgfmathsetmacro\num{int(4*\k+2*\j+\i)}
        \pgfmathsetmacro\x{16*\i+8*\j+4*\k+2}
        \pgfmathsetmacro\c{4*\x}
        \colorlet{numcolor\k\j\i}{black!35!red!\c!blue}
        \draw (\s*\x,3) node[treenode,scale=.24,numcolor\k\j\i, very thin] (\k\j\i) {$\k\j\i$};
        \ifthenelse{\k=1}{\draw[numcolor\j\i,bend right=\s*3] (\j\i) to (\k\j\i);
                          \draw[numcolor\k\j\i] (\k\j\i)+(0,3) node[scale=.49] (number\k\j\i) {$\ifthenelse{\s=1}{}{\unaryminus}\num$};
                          \draw[toparrow,numcolor\k\j\i,very thin] (\k\j\i) -- (number\k\j\i);}
                         {\draw[numcolor\j\i,thin,bend left=\s*3] (\j\i) to (\k\j\i);}
        \foreach \l in {0,1}
          {
          \pgfmathsetmacro\num{int(8*\l+4*\k+2*\j+\i)}
          \pgfmathsetmacro\x{16*\i+8*\j+4*\k+2*\l+1}
          \pgfmathsetmacro\c{4*\x}
          \colorlet{numcolor\l\k\j\i}{black!35!red!\c!blue}
          \draw (\s*\x,4) node[treenode,scale=.168,numcolor\l\k\j\i,ultra thin]
                                                         (\l\k\j\i) {$\l\k\j\i$};
          \ifthenelse{\l=1}{\draw[numcolor\k\j\i,very thin,bend right=\s*3] (\k\j\i) to (\l\k\j\i);
                            \draw[numcolor\l\k\j\i] (\l\k\j\i)+(0,2) node[scale=.343] (number\l\k\j\i) {$\ifthenelse{\s=1}{}{\unaryminus}\num$};
                            \draw[toparrow,numcolor\l\k\j\i,very thin] (\l\k\j\i) -- (number\l\k\j\i);}
                           {\draw[numcolor\k\j\i,very thin,bend left=\s*3] (\k\j\i) to (\l\k\j\i);}
          }
        }
      }
    }
  }
\end{tikzpicture}
\caption{The signed-final-digits order on the integers}\label{Figure.Signed-final-digits-order}
\end{figure}

 By carrying out this argument inside any countable nonstandard model of $M$ of $\text{\rm I}\Delta_0$ , we conclude that $\Z^M$ also has a continuous presentation on $\Q^M$, which is order-isomorphic and hence homeomorphic to $\Q$.
\end{proof}

\begin{remark}

\rm{Let us briefly discuss a variant of the final-digits order, which some readers might find natural, but which generates a different topology from the final-digits topology. The order, pictured in figure~\ref{Figure.Final-digits-variation} arises from the lexical order on infinitary binary sequences, if one embeds the finite binary sequences into that space by appending infinitely many zeros at the left. Thus, one imagines representing the number $6$ with $\cdots0000110$. This amounts to treating missing initial digits in the usual binary representations as $0$.}
\begin{figure}[h]
\begin{tikzpicture}[scale=.35,xscale=1.2,>=Stealth,treenode/.style={circle,draw},toparrow/.style={dotted}]
\colorlet{numcolor}{red!0!blue!65!black}
\draw[numcolor] (0,0) node[treenode,scale=1,very thick] (root) {\phantom{$0$}};
\draw[numcolor] (root)+(0,12) node[scale=1.6,thick] (number0) {$0$};
\draw[toparrow,numcolor,thick] (root)+(0,10.5) -- (number0);
\foreach \i in {0,1}
  {
  \pgfmathsetmacro\num{int(\i)}
  \pgfmathsetmacro\c{4*(16*\i)}
  \colorlet{numcolor\i}{red!\c!blue!65!black}
  \draw (16*\i,2) node[treenode,scale=.9,numcolor\i,thick] (\i) {$\i$};
  \ifthenelse{\i=1}{\draw[numcolor,thick,bend right=4] (root) to (\i);
                    \draw[numcolor\i] (\i)+(0,10) node[scale=1.4] (number\i) {$1$};
                    \draw[toparrow,numcolor\i,thick] (\i) +(0,8.5)-- (number\i);}
                   {\draw[numcolor,thick] (root) -- (\i);}
  \foreach \j in {0,1}
    {
    \pgfmathsetmacro\num{int(2*\j+\i)}
    \pgfmathsetmacro\c{4*(16*\i+8*\j)}
    \colorlet{numcolor\j\i}{red!\c!blue!65!black}
    \draw (16*\i+8*\j,4) node[treenode,scale=.49,numcolor\j\i] (\j\i) {$\j\i$};
    \ifthenelse{\j=1}{\draw[numcolor\i,bend right=4] (\i) to (\j\i);
                      \draw[numcolor\j\i] (\j\i)+(0,8) node[scale=1] (number\j\i) {$\num$};
                      \draw[toparrow,numcolor\j\i] (\j\i) +(0,6.5)-- (number\j\i);}
                     {\draw[numcolor\i] (\i) -- (\j\i);}
    \foreach \k in {0,1}
      {
      \pgfmathsetmacro\num{int(4*\k+2*\j+\i)}
      \pgfmathsetmacro\c{4*(16*\i+8*\j+4*\k)}
      \colorlet{numcolor\k\j\i}{red!\c!blue!65!black}
      \draw (16*\i+8*\j+4*\k,6) node[treenode,scale=.343,numcolor\k\j\i,thin] (\k\j\i) {$\k\j\i$};
      \ifthenelse{\k=1}{\draw[numcolor\j\i,thin,bend right=4] (\j\i) to (\k\j\i);
                        \draw[numcolor\k\j\i] (\k\j\i)+(0,6) node[scale=.7] (number\k\j\i) {$\num$};
                        \draw[toparrow,numcolor\k\j\i,thin] (\k\j\i) +(0,4.5) -- (number\k\j\i);}
                       {\draw[numcolor\j\i,thin] (\j\i) -- (\k\j\i);}
      \foreach \l in {0,1}
        {
        \pgfmathsetmacro\num{int(8*\l+4*\k+2*\j+\i)}
        \pgfmathsetmacro\c{4*(16*\i+8*\j+4*\k+2*\l)}
        \colorlet{numcolor\l\k\j\i}{red!\c!blue!65!black}
        \draw (16*\i+8*\j+4*\k+2*\l,8) node[treenode,scale=.24,numcolor\l\k\j\i,very thin]
                                                       (\l\k\j\i) {$\l\k\j\i$};
        \ifthenelse{\l=1}{\draw[numcolor\k\j\i,very thin,bend right=4] (\k\j\i) to (\l\k\j\i);
                          \draw[numcolor\l\k\j\i] (\l\k\j\i)+(0,4) node[scale=.49] (number\l\k\j\i) {$\num$};
                          \draw[toparrow,numcolor\l\k\j\i,very thin] (\l\k\j\i)+(0,2.5) -- (number\l\k\j\i);}
                         {\draw[numcolor\k\j\i,very thin] (\k\j\i) -- (\l\k\j\i);}
        \foreach \m in {0,1}
          {
          \pgfmathsetmacro\num{int(16*\m+8*\l+4*\k+2*\j+\i)}
          \pgfmathsetmacro\c{4*(16*\i+8*\j+4*\k+2*\l+\m)}
          \colorlet{numcolor\m\l\k\j\i}{red!\c!blue!65!black}
          \draw (16*\i+8*\j+4*\k+2*\l+\m,10) node[treenode,scale=.168,numcolor\m\l\k\j\i,ultra thin]
                                (\m\l\k\j\i) {$\m\l\k\j\i$};
          \ifthenelse{\m=1}{\draw[numcolor\l\k\j\i,ultra thin,bend right=4] (\l\k\j\i) to (\m\l\k\j\i);
                            \draw[numcolor\m\l\k\j\i] (\m\l\k\j\i)+(0,2) node[scale=.343] (number\m\l\k\j\i) {$\num$};
                            \draw[toparrow,numcolor\m\l\k\j\i,very thin] (\m\l\k\j\i) -- (number\m\l\k\j\i);}
                           {\draw[numcolor\l\k\j\i,ultra thin] (\l\k\j\i) -- (\m\l\k\j\i);}
          }
        }
      }
    }
  }
\end{tikzpicture}
\caption{A variant of the final-digits order on the natural numbers}\label{Figure.Final-digits-variation}
\end{figure}
\rm{The order is manifestly different from the final-digits order; for example, $0$ is now a least element. Nevertheless, this is a dense linear order with a least element, and so it is order-isomorphic to $\Q\intersect[0,1)$, which is homeomorphic to $\Q$. We should like to point out, however, first, that this order does not give rise to the final-digits topology, since with this variant order, the set of numbers having final digits $110$ has a least element, the number $6$, and so it is not open in this order topology. Secondly, the successor function $x\mapsto x+1$ is not continuous with respect to this order, since $1+1=2$, and $2$ has an open neighborhood containing only even numbers, but every open neighborhood of $1$ contains both even and odd numbers, whose sums therefore will not all be in that neighborhood of $2$. Thus, even the successor function is not continuous with respect to this modified order, and in this sense, it was not the right order to consider. Nevertheless, one can rescue things somewhat by noting that the final-digits topology is exactly the right-open topology of this order, with basic open sets consisting of half-open intervals $[n,m)$ in the order here, and therefore, by the argument of the main theorem, arithmetic is continuous with respect to the right-open topology of this order.}
\end{remark}
\begin{corollary}\label{Corollary.Continuous-on-Q^M}
 Every model $M$ (of any cardinality) of $\text{\rm I}\Sigma_1$ has a continuous presentation on the topological space $\Q^M$ with the order topology.
\end{corollary}

\begin{proof}
The model $M$ has a continuous presentation with respect to the final-digits order as defined in $M$. This order is seen as a countable endless dense linear order, from the perspective of $M$, and so by an appropriate version of Cantor's theorem applied in $M$, it is order-isomorphic and hence homeomorphic to the rationals $\Q^M$ of $M$. More explicitly, $M$ can handle constructions by $\Sigma_1$-recursion since $M\satisfies \text{\rm I}\Sigma_1$, and the ordering on $\Q^M$ and the final digits order $\fdlt$ on $M$ are both $\Delta_1$-definable in $M$, therefore Cantor's back-and-forth construction can be implemented within $M$ to yield a definable isomorphism between the endless dense linearly ordered sets $\Q^M$ and $(M,\fdlt)$.
\end{proof}

The investigation of saturated dense linear orders was initiated by Hausdorff (see~\cite{Hausdorff2002:GesammelteWerke}), who introduced the familiar $\eta$ notation for the countable endless dense linear order, and more generally $\eta_\alpha$ for the saturated endless dense linear order of cardality $\aleph_\alpha$, which exists precisely under the set-theoretical assumption that $\aleph_\alpha^{\smalllt\aleph_\alpha}=\aleph_\alpha$.

\goodbreak
\begin{corollary}
 The order $\eta_\alpha$, when it exists, supports a topological model of arithmetic.
\end{corollary}

\begin{proof}
If $\eta_\alpha$ exists, then because $\aleph_\alpha^{\smalllt\aleph_\alpha}=\aleph_\alpha$, it follows that there are saturated structures of size $\aleph_\alpha$ in any first-order language of size less than $\aleph_\alpha$. Let $M$ be such a saturated model of arithmetic of size $\aleph_\alpha$. Its rational numbers $\Q^M$ will therefore be a saturated endless dense linear order of size $\aleph_\alpha$, hence order-isomorphic to $\eta_\alpha$. So $\eta_\alpha$, when it exists, supports a topological model of arithmetic.
\end{proof}

\begin{question}
Which topological spaces arise as $\Q^M$ for a model of arithmetic $M\satisfies\PA$?
\end{question}

The spaces $\Q^M$ are always at least a little homogeneous, since there are many order-isomorphisms of $\Q^M$ internal to $M$; for example, any finite partial order automorphism extends inside $M$ to a full order automorphism. Also, $\Q^M$ is never complete as an order, since the model can see that $\sqrt{2}^M$ determines a cut in $\Q^M$ with no least upper bound; indeed, the order is totally disconnected. The following theorem collects some basic facts about linear orders definable in nonstandard models of arithmetic.

\begin{theorem} \label{basic order properties} Let $M$ be a nonstandard model of \PA. \begin{enumerate}

\item [\rm{(a)}]  No infinite complete linear order is definable in $M$.

\item [\rm{(b)}] If a linear order $\<L,\leq_L>$ is definable in $M$, then every limit point in $L$ has the same cofinality (from below or above) as the cofinality of $M$, and this is furthermore the same as the cofinality or coinitiality of $L$ itself.

\item [\rm{(c)}] If $M$ is uncountable, then $\Q^M$ is not separable, and indeed, it has an uncountable discrete subset, and therefore admits an uncountable family of disjoint intervals.
\footnote{In particular, $\Q^M$ is never a Suslin line, nor is it the space of irrational numbers. While part (c) of Theorem \ref{basic order properties} rules out these latter spaces for $\Q^M$, it does not seem necessarily to rule them out as supporting topological models of arithmetic in some way other than as $\Q^M$.}
\end{enumerate}
\end{theorem}

\begin{proof} (a)
Suppose that $\<L,\leq_L>$ is a linear order definable in a nonstandard model of arithmetic $M\satisfies\PA$. Since Ramsey's theorem is formalizable in \PA, from the perspective of $M$ there must either be an infinite increasing or decreasing sequence $\<c_i\mid i\in M>$. But the limit of $c_k$ for standard $k$ must exist, if $L$ is complete, and from this $M$ can define the standard cut, which is a contradiction. \medskip

(b)   If a point $x$ is a limit point of $L$ in $M$, then $M$ thinks that $x$ is the limit of an infinite sequence, and so $x$ is the limit of a sequence whose order type is the same as the order type of $M$. Thus, the cofinality of $x$ in $L$ (from either direction) is the same as the cofinality of $M$ itself. The same argument works for the cofinality and coinitiality of $L$ itself, if the order does not have maximal or minimal elements, respectively. \medskip

(c) Every model $M$ of \PA\ sits as a discrete set inside its rationals $\Q^M$. The model can place disjoint intervals around its integers, such as $(n-\frac13,n+\frac13)$, and thereby produce an uncountable family of disjoint intervals.

\end{proof}

\section{Some spaces do not support a topological model of arithmetic}\label{Section.Non-examples}

In this section we prove now that several types of spaces do not support a topological model of arithmetic. We begin with the case of compact Hausdorff spaces. The proof of the theorem below shows that if $M$ is a compact Hausdorff space and $\<M,+,\cdot>$ is a model of a sufficiently strong fragment of \PA\ then neither addition nor multiplication, separately, is continuous, and indeed, these functions are not even continuous separately in just one coordinate at a time.

\goodbreak

\begin{theorem}\label{Theorem.No-compact-Hausdorff-space}
 No compact Hausdorff space supports a topological model of arithmetic. More precisely:
 \begin{enumerate}
 \item [\rm{(a)}] Addition is not continuous for any compact Hausdorff topological model of Presburger Arithmetic.

 \item [\rm{(b)}] Multiplication is not continuous for any compact Hausdorff topological model of $\text{\rm I}\Delta_0+\mathrm{Exp}$.

 \end{enumerate}

 \end{theorem}

\begin{proof}
(a) Suppose that $M$ is a model of Presburger Arithmetic. To show that addition is not continuous, for any $k\in M$, consider the interval $[k,\infty)$ in the sense of $M$. This interval is the image of $M$ under the map $x\mapsto x+k$. If this map is continuous, then the image is compact and therefore also closed in $M$, and so the family of all such intervals is thus a nested family of nonempty closed sets with empty intersection, violating compactness. So addition is not continuous, not even separately in the first coordinate. \medskip

(b) Let $M$ be a model of $\text{\rm I}\Delta_0+\text{Exp}$. Consider the sets $M\cdot(k!)=\set{n\cdot(k!)\mid n\in M}$, where $k!$ means the factorial of $k$ as seen by the model $M$ (the factorial function is well-defined within $\text{\rm I}\Delta_0+\text{Exp}$). If multiplication is continuous, then these sets are the image of $M$ under a continuous map and therefore are compact and hence closed. And they are nested, as $k$ increases in $M$, with empty intersection. So we have a nested family of closed sets with empty intersection, contradicting compactness. So multiplication is not continuous, not even separately in the first coordinate.

\end{proof}

\begin{theorem}\label{ruling out euclidean spaces} No Euclidean space $\R^n$, in any finite dimension $n \ge 1$, supports a topological model of Successor Arithmetic.
\end{theorem}

\begin{proof} The $n=1$ case is straightforward. Suppose that $\<\R,r_0,S(x)>$ is a topological model of Successor Arithmetic, where $r_0$ is the zero of the model (which need not equal the real number 0). Also note that $S(x)$ is the successor function in the sense of the model of arithmetic, which need not equal the real number $x+1$. If $S(x)$ were continuous, then the image of $\R$ under this map would be connected, since $\R$ is connected. But the complement of a point in $\R$ is disconnected and Successor Arithmetic proves that the range of the successor function omits exactly one point, namely $r_0$. Therefore $S(x)$ is discontinuous. \medskip

Now suppose that $M$ is a topological model of Successor Arithmetic that is homemorphic to $\R^n$ for some finite $n>1$. Then the continuity of the successor function implies that the punctured space $\R^n \setminus \{r_0\}$ is a continuous injective image of $\R^n$. On the other hand, by Brouwer's ``Invariance of Domain Theorem'' ~\cite[section 62]{Munkres2000:Topology}, if $U$ is an open subset of $\R^n$ and $f:U \rightarrow \R^n$ is injective and one-to-one, then $f[U]$ is homeomorphic to $U$. This shows that the continuity of the successor function implies that $\R^n$ is homeomorphic to the punctured space $\R^n \setminus \{r_0\}$, which is well-known to be false; one standard way to verify this is to note that the $n$-th homotopy group $\pi_n (\R^n)$  is the trivial group since $\R^n$ is contractible (i.e., homotopy equivalent to a point), but $\pi_n (\R^n\setminus \{r_0 \}) = \mathbb{Z}$ since $\R^n \setminus \{r_0\}$ is homotopy equivalent to the $n$-dimensional sphere $S_n$, and by a basic result of algebraic topology $\pi_n (S_{n})= \mathbb{Z}$, cf.~\cite [corollary 4.25]{Hatcher:Algebraic Topology}.
\end{proof}
The following lemma can also be used to rule out finite dimensional Euclidean spaces $\R^n$ based on considerations of
dimension theory in topology that show that there can be no continuous injection of $\R^m$ into $\R^n$, when $n<m$; such a continuous injection would provide a homeomorphism of the $m$-dimensional cube $[0,1]^m$ with its image in $\R^n$, which violates~\cite[theorem~50.6]{Munkres2000:Topology}, stating that closed subsets of $\R^n$ have dimension at most $n$. So $M$ cannot have $\R^n$ as its underlying space.
\begin{lemma}\label{Lemma.Pairing-function}
  If $M$ is a topological model of $\mathrm{IOpen}$, then there is a continuous pairing function, that is, a continuous injection from $M^2$ to $M$. Indeed, there is a continuous bijection of $M^2$ with the even numbers of $M$.
\end{lemma}

\begin{proof}
Provably in $\mathrm{IOpen}$, the Cantor pairing function $p(n,m)=\frac 12(n+m)(n+m+1) +m$ is a bijection between $M^2$ and $M$. We cannot generally divide by $2$ in a continuous manner, however, and so to make a continuous injection, as opposed to bijection, we simply use $q(n,m)=(n+m)(n+m+1)+2m$, which is the double of the Cantor function. This is now a continuous injection of $M^2$ into $M$, and indeed, it is a continuous bijection of $M^2$ with the even numbers of $M$.

\end{proof}

\begin{remark}
\rm{The argument for ruling out $\R$ in theorem \ref{ruling out euclidean spaces} shows that no space that is connected, yet becomes disconnected upon the deletion of any point (or any two points), supports a topological model of arithmetic. Therefore the two-way version of the long line, which extends to $\omega_1$ in both directions, does not support a topological model of successor arithmetic.  For the ordinary long line, that is, $[0,1)\cdot\omega_1$, there is a least element, whose deletion does not disconnect the space, although the deletion of any two points will disconnect the space. This again violates continuity of the successor function, if we simply perform it twice, if necessary, for then the image will omit exactly two points.}
\end{remark}

The following theorem rules out another class of spaces which includes the real line (but not the long line). Indeed, this was the first proof we found for ruling out the real line.

\begin{theorem}If a space is uncountable and connected, and has the further properties that (i) there is no descending $\omega_1$-sequence of connected subspaces, and (ii) there is no partition of the space into uncountably many disjoint uncountable connected subspaces, then it does not support a topological model of $\mathrm{I}\Delta_0+\mathrm{Exp}$.
\end{theorem}
\begin{proof}
Suppose that $\<M,+,\cdot>$ is an uncountable topological model of $\mathrm{I}\Delta_0+\mathrm{Exp}$, whose underlying space has the features mentioned in the theorem. Recall that the factorial function and the division algorithm can be readily handled in $\mathrm{I}\Delta_0+\mathrm{Exp}$. We split into two cases. First, suppose that the model happens to have an increasing $\omega_1$-sequence of numbers $\<k_\alpha\mid\alpha<\omega_1>$. (But see theorem \ref{Theorem.Polish-space-no-uncountable-chains} for why this case does not arise in $\R$ or in any Polish space.) We claim that in this case multiplication cannot be continuous, even just in one coordinate. To see this, consider the image of the whole space $M\cdot(k_\alpha!)$ after  multiplication by the factorial of $k_\alpha$, as the model sees it. This is the continuous image of a connected set and therefore connected. Since these sets are descending as $k_\alpha$ increases, this would violate assumption (i) on the underlying space of $M$. In the remaining case, since $M$ is uncountable, there must be a number $d$ in the model having at least $\omega_1$ many predecessors $r<d$. For each such $r$, let $A_r$ be the numbers with residue $r$ modulo $d$ as the model sees it. Thus, $A_r$ is simply the image of $M$ under the function $q\mapsto dq+r$, which is injective and continuous. So each $A_r$ is uncountable and connected, and we have therefore found a partition of $M$ into uncountably many disjoint uncountable connected sets, which would violate assumption (ii) on the underlying space of $M$.
\end{proof}

\begin{question}\label{unsolved question}
 Does the Baire space support a topological model of \PA? Does $\R^\omega$ support a topological model of \PA? Does any uncountable Polish space support a topological model of $\PA$?

\end{question}

\begin{remark} \label{Motivation for Main Question}
\rm{

(a) The above question, initially posed in ~\cite{Enayat2009:Borel-structures-via-models-of-arithmetic}, was motivated by a key result of \cite{MalitzMycielskiReinhardt1991:The-axiom-of-choice-the-Lowenheim-Skolem-theorem-and-Borel-models,
MalitzMycielskiReinhardt1992:Erratum-to-AC-LS-and-Borel-models}, that shows that every consistent first order theory in a countable language has a model whose universe is the Baire space and all of whose definable relations are $F_\sigma$ as well as $G_\delta$. So we know that there are low-complexity Borel models on the Baire space and other Polish spaces; the question is whether one can realize the model with continuous functions.\footnote{Note that continuous functions in Polish spaces have closed graphs, or in other words, in the class $F$. In compact Hausdorff spaces, this is an equivalence, and so theorem \ref{Theorem.No-compact-Hausdorff-space} shows that in compact Hausdorff spaces, there is no model of arithmetic whose operations have a closed graph. Does this extend to $\R$ and the other spaces we have ruled out? Will Brian's observation \cite{Brian.MO307419:Can-an-injective-f:R^m->R^n-have-a-closed-graph}, posted on MathOverflow in answer to a question of the first author, may be relevant; he shows that there are no injections from $\R^m$ to $\R^n$ with a closed graph, when $m>n$; but note that the composition of functions with a closed graph need not necessarily have a closed graph.}  In light of the homeomorphism between the Baire space and the space of irrationals (whose topology is induced by a linear order), the L{\"o}wenheim-Skolem theorem implies that if the Baire space supports a topological model of arithmetic, then any countable subset $X$ of the Baire space is contained in a countable subspace $Y$ of the Baire space that continuously supports a model of arithmetic. Putting this fact together with Cantor's theorem concerning the countable categoricity of endless dense linear orders, one obtains the implication: If the Baire space continuously supports a model of arithmetic, then so does the space of rationals $\Q$. This explains the motivation for Main Question 1 since a negative answer to it would have automatically provided a negative answer to question \ref{unsolved question} (for the Baire space).

\medskip

(b) It is easy to construct uncountable Polish spaces that support topological models of Successor Arithmetic. For example, let $X$  be the union of $\mathbb{N\times \{}0\}$ and $\mathbb{Z}\times \mathbb{I}$, where $\mathbb{I}$ is the interval $[1,2]$. $X$ is a closed subspace of $\R^2$, and therefore is an uncountable Polish space. Let $f:X\rightarrow X$ be the restriction to $X$ of the map $F(x,y)=(x+1,y)$. Clearly $f$ is continuous on $X$ since $F$ is continuous on $\R^2$. It is routine to verify that $\left\langle X,z,f\right\rangle $ is a model of Successor Arithmetic with the choice of $z=(0,0)$, i.e., $f$ is injective, the range of $f$ is $X \setminus \{ z \}$, and $f^{(n)}(x) \neq x$ for all $x \in X$ and all $n \geq 1$, where  $f^{(n)}$ is the $n$-fold iteration of the function $f$.  With slightly more effort, one can show that the subspace $Y$ of $\R$ consisting of the union of $\{(n-1)/n: n \in \N \}$  and the open interval $(1,2)$ also supports a topological model of Successor Arithmetic. Note that $Y$ is $G_{\delta}$ in $\R$ and is therefore a Polish space.

}
\end{remark}

The following observation shows that certain well-studied uncountable models of arithmetic cannot be continuously presented on Polish spaces.

\goodbreak
\begin{theorem}\label{Theorem.Polish-space-no-uncountable-chains}
 If a model of arithmetic $M$ of Presburger Arithmetic is continuously presented on an uncountable Polish space, then there are no increasing or decreasing $\omega_1$-sequences in the order. In particular, it is not $\kappa$-like for any cardinal $\kappa$.
\end{theorem}

\begin{proof}
If addition $+$ is continuous, then the order $<$ is analytic, being the projection of a Borel relation via $a<b\longleftrightarrow\exists c\ (a+c+1=b)$. But a deep theorem due to Harrington and Shelah~\cite{HarringtonShelah1982:Counting-equivalence-classes-for-co-kappa-Souslin-equivalence-relations} shows that no analytic order can have increasing or decreasing  $\omega_1$-sequences.
\end{proof}

\section{The topology of a topological model of arithmetic}

We should like to focus now on the underlying topology of a topological model of arithmetic, finding general features these models exhibit and investigating how one might change the topology to another related topology with respect to which the model remains continuous.

\begin{theorem}\label{Theorem.Discrete-indiscrete-initial-segment-final-segment}
 Every model $M$ of a sufficiently strong fragment of \PA ~admits a continuous presentation with respect to the following topologies on $M$.
 \begin{enumerate}
   \item [\rm{(a)}] The discrete and indiscrete topologies on $M$.
   \item [\rm{(b)}] The rational line $\Q^M$, as $M$ sees it.
   \item [\rm{(c)}] The initial-segment topology, whose open sets are initial segments of $M$.
   \item [\rm{(d)}] The final-segment topology, whose open sets are final segments of $M$.
 \end{enumerate}

\end{theorem}

\begin{proof}
Statement (a) is immediate, since all operations on a set are continuous with respect to discrete or indiscrete topologies. Statement (2) follows from the main theorem, so ``sufficiently strong fragment'' for statement (b) means $\mathrm{I}\Delta_0$. For statement (c), consider the initial-segment topology. If $x+y\leq k$, then for any $x'\leq x$ and $y'\leq y$, we would have $x'+y'\leq k$. Thus, we have an open neighborhood of $(x,y)$ whose image under addition is contained in the open set $[0,k]$, and so addition is continuous. A similar argument works for multiplication: if $xy\leq k$, then any $x'\leq x$ and $y'\leq y$ will have $x'y'\leq k$. So we have verified statement (c). For statement (d), we consider the final-segment topology. If $x+y\geq k$, then for any $x'\geq x$ and $y'\geq y$, we will have $x'+y'\geq k$, and so addition is continuous. Similarly, if $xy\geq k$, then for any $x'\geq x$ and $y'\geq y$, we have $x'y'\geq k$, and so multiplication is continuous, verifying statement (4). Thus, for statements (3) and (4) ``sufficiently strong fragment" means $\PA^-$.
\end{proof}

The key for statements (3) and (4) was the monotone nature of addition and multiplication in a model of arithmetic; the argument would not work (and the result fails) in ordered rings, where multiplication by negative numbers is order reversing. Note that those topologies are not Hausdorff.

\begin{lemma}\label{Lemma.Upward-homogeneity}
 Suppose $M$ is a topological model of Presburger arithmetic. Then:
 \begin{enumerate}
  \item [\rm{(a)}] $M$ is upward homogeneous, meaning that whenever $a<b$ in $M$, then there is a continuous injection on $M$ taking $a$ to $b$. In particular, if $b$ is isolated, then also all $a<b$ are isolated.
  \item [\rm{(b)}] If $M$ is homeomorphic to the rationals $\Q$, then for each $a \in M$, $\{ x \in M: x>a \}$ is also homeomorphic to $\Q$.
  \end{enumerate}
\end{lemma}

\begin{proof}
For (a) simply use the map $x\mapsto x+k$, where $k=b-a$ in $M$. (b) follows from (a) together with \Sierpinski 's characterization of $\Q$ as the only countable metric space (up to homeomorphism) with no isolated points.
\end{proof}

Topological models of arithmetic might fail to have the corresponding downward homogeneity, however, in light of the following.

\begin{theorem}\label{Theorem.Making-points-isolated}
 Suppose that $M$ is a topological model of $\PA^-$, and consider any element $b\in M$. Let $M^*$ be the same model of arithmetic, but refining the topology to make every number $a\leq b$ isolated in $M^*$. Then $M^*$ is continuous with respect to this new topology.
\end{theorem}

\begin{proof}
The open sets of $M^*$ have the form $U^*=U\union A$, where $U$ is open in $M$ and $A\of[0,b]$. To verify continuity in $M^*$, suppose that $x+y\in U^*=U\union A$. If $x+y\leq b$, then both $x$ and $y$ also are $\leq b$, and so $(x,y)$ is isolated in the plane, which is sufficient for this instance of continuity. Otherwise, we have $x+y\in U$, and so by the continuity of $M$ there is a neighborhood of $(x,y)$ in $M^2$ whose image under addition is contained in $U$. Thus, addition is continuous in $M^*$. A similar argument works for multiplication, with a slight wrinkle about zero. Namely, if $b\neq 0$ and $xy\leq b$, then both $x,y\leq b$ and we may proceed as with addition. And when $xy=0$, then one of them must be $0$, and in this case, $M\times\{0\}$ is an open set in the product space of $M^*$ whose image maps to $0$ under multiplication. So again we have verified continuity.
\end{proof}

Since isolating any set of points in a metric space preserves metrizability, it follows that $M^*$ will be metrizable, if $M$ was. One can view the construction of theorem \ref{Theorem.Making-points-isolated} as an instance of the more general topology-blending constructions, which we now develop. These constructions are generally asymmetric, in that one is either adding open sets of small numbers or removing open sets of large numbers.

\begin{theorem}\label{Theorem.Topology-blending}
 Suppose that $M$ is a model of $\PA^-$ that is continuous with respect to topologies $\tau_0$ and $\tau_1$, where $\tau_0$ refines $\tau_1$. For any initial segment $I\of M$, let $\tau^*$ be the topology generated by sets of the form $U\union A$, where $U\in\tau_1$ and $A$ is an open subset of $I$ in the subspace topology of $\tau_0$. Then $M$ is continuous with respect to the topology $\tau^*$.
\end{theorem}

\begin{proof}
If $x+y\in U\union A$, then we break into cases. If $x+y\in U$, then we can find a neighborhood of $(x,y)$ using the $\tau_1$ topology. Otherwise $x+y\in A=W\intersect I$ for some $W\in\tau_0$. In particular, $x+y\in I$ and so also $x,y\in I$. Thus, we can find a neighborhood of $(x,y)$ using the $\tau_0$ topology, and then restrict to the subspace $I$, which will give a neighborhood in the $\tau^*$ product topology. A similar argument works with $xy$, where again we need to pay attention to the case $xy=0$ separately.
\end{proof}

\begin{corollary}\label{Corollary.Tau-restrict-b}
 If $M$ is a topological model of $\PA^-$ with topology $\tau$, then for any initial segment $I\of M$, the model $M$ is also continuous with respect to the topology $\tau\restrict I$, whose open sets are of the form $U\intersect I$, where $U$ is open with respect to $\tau$, plus $M$ itself.
\end{corollary}

\begin{proof}
The topology $\tau\restrict I$ is an instance of theorem~\ref{Theorem.Topology-blending}, blending the indiscrete topology with the topology $\tau$ on $I$.
\end{proof}

\begin{remark} \label{Remark on ordered rings}
\rm{
Under the topology $\tau\restrict I$, the initial segment $I$ appears the same as its subspace topology with respect to $\tau$, but the points above $I$ in the order of $M$ become indiscrete. This method can be used even in the case $I=[0,b]$, where the initial segment has a maximal element, in which case we denote $\tau\restrict [0,b]$ by $\tau\restrict b$. Because of these topology-blending constructions, it follows that topological models of arithmetic need not be homogeneous. This makes the case of topological models of arithmetic very different from the related case of topological ordered rings, such as the integers $\Z^M$ of a model of arithmetic $M$, as these are always homogeneous: the translation maps $x\mapsto x+k$ are homeomorphisms mapping any point to any other desired point (the inverse map is continuous, since it is translation by $-k$). In particular, in such a topological ordered ring, if one point is isolated, then all points are isolated. It follows that if $\<Z,+,\cdot>$ is a countable ordered ring, whose addition $+$ is continuous with respect to an underlying metric topology, then the underlying space is homeomorphic either to $\N$ or $\Q$. This is because if there is an isolated point, then all points are isolated and the topology is discrete; and otherwise, there are no isolated points, and by the theorem of \Sierpinski\ we mentioned earlier, every countable metric space without isolated points is homeomorphic to $\Q$. With topological models of arithmetic, in contrast, one can achieve other countable metric spaces.
}
\end{remark}

\begin{theorem}\label{Theorem.Countable-metric-space-presentations}
Every countable nonstandard model of $\text{\rm I}\Delta_0$ has continuous presentations with respect to the following countable metric spaces:
\begin{enumerate}
  \item[\rm{(a)}] The rationals $\Q$.
  \item[\rm{(b)}] The rationals $\Q$, adjoined with finitely many isolated points.
  \item[\rm{(c)}] The rationals $\Q$, adjoined with infinitely many isolated points.
  \item[\rm{(d)}] The countable discrete metric space.
\end{enumerate}
\end{theorem}

\begin{proof}
The main theorem gives $\Q$, and then by theorem \ref{Theorem.Topology-blending} we can choose any finite or infinite initial segment $I$ of the model and arrange for all points in $I$ to be isolated. Note that by part (b) of Lemma \ref{Lemma.Upward-homogeneity}, the rest of the model is homeomorphic to $\Q$.
\end{proof}

Note that the standard model of arithmetic has no continuous presentation on $\Q$ adjoined with infinitely many isolated points, since those points would have to be unbounded in $\N$ and thus every point would have to be isolated, by lemma \ref{Lemma.Upward-homogeneity}.

\medskip

The underlying topology-blending fact, of course, is the following elementary observation of topology.

\begin{lemma}
 If a binary operation on a set $X$ is continuous with respect to topologies $\sigma$ and also with respect to $\tau$, then it is also continuous with respect to the topology generated by $\sigma\union\tau$.
\end{lemma}

\begin{proof}
Suppose that $\star$ is a binary operation on $X$, continuous with respect both to $\sigma$ and $\tau$ separately, meaning that the corresponding product topology is used for $X^2$, the domain of the operation. The topology generated by $\sigma\union\tau$ has basic open sets of the form $U\intersect W$, where $U\in\sigma$ and $W\in\tau$.  If $x\star y\in U\intersect W$, then we can find open neighborhoods of $(x,y)$ in the $\sigma$ topology on $X^2$ and in the $\tau$ topology on $X^2$, which by intersection will provide an open neighborhood in the combined topology, whose image under $\star$ is contained in $U\intersect W$.
\end{proof}

For example, we may combine any given topology with the initial-segment topology, or with the final-segment topology, and produce a new topological model of arithmetic.

\begin{theorem}\label{Theorem.Initial/final-segment-topology}
 If $M$ is a topological model of $\PA^-$, then so is $M^*$, which is the same model, but with a new topology, generated from the old by taking all initial segments of $M$ also as open sets. The same applies to $M^{**}$, whose topology is generated by the topology of $M$ and the final segments of $M$.
\end{theorem}

Note that the final-digits topology on a model of arithmetic $M$, as in the main theorem, already makes all final-segments open, since for any $n$ with binary representation $s$, the basic open set $U_s$ consists of numbers all at least as large as $n$. So if one augments $M$ with the initial segments to form $M^*$, then every point will become isolated.

Let us discuss the issue in the proof of theorem~\ref{Lemma.Pairing-function} about using the Cantor bijective pairing function
$$p(n,m)=\frac 12(n+m)(n+m+1) +m,$$
instead of its double. If dividing by $2$ were continuous, then indeed we could use the bijective pairing function, and $M^2$ would be continuously bijective with $M$. But not every topological model of arithmetic has dividing-by-$2$ as a continuous function.

\goodbreak

\begin{theorem}\label{Theorem.Dividing-by-2-discontinuous} Every model $M$ of $\PA^-$ carries a topology $\tau_{M}$ that continuously supports $M$ in which dividing by $2$ is not a continuous function on its even numbers, and also in which the predecessor function $x\mapsto x-1$ is not continuous on the positive numbers. Moreover, if $M$ is a countable model of $\mathrm{I}\Delta_0$, then $\tau_{M}$  can be arranged to be metrizable.

\end{theorem}

\begin{proof}
Let $M$ be any model of arithmetic. Consider this model under the topology $\tau_{M} = \tau\restrict 17$, where $\tau$ is the discrete topology. So $M$ has the discrete topology up to and including the number $17$, but above $17$ the topology is indiscrete. By corollary~\ref{Corollary.Tau-restrict-b}, this is a topological model of arithmetic.

If dividing by $2$ were a continuous function on the even numbers of $M$, then since $30$ gets mapped to $15$, which is isolated, it would mean that $30$ is isolated in the subspace topology of the even numbers. But the only open set containing $30$ also contains all the numbers above $17$, and so it isn't isolated in the even numbers. Contradiction.

A similar argument works with the predecessor function. Since $18-1=17$, if this function were continuous on the positive numbers, then $18$ would have to be isolated in the subspace of positive numbers, but it isn't.

In order to demonstrate the moreover clause, recall that by the proof of theorem \ref{Theorem.Countable-metric-space-presentations} there is a model $M$ of $\mathrm{I}\Delta_0$ that carries a metrizable topology $\tau_{M}$ that continuously supports $M$ and such that the subspace consisting of numbers less than or equal to 17 carries the discrete topology, but the rest of the space is homeomorphic to $\Q$. An argument similar to the one used to prove the first part of the theorem shows that when $M$ is equipped with $\tau_{M}$, the map $x\mapsto 2x$ is not an open map, and neither is the successor function $x\mapsto x+1$, since open maps would map isolated points to isolated points. Therefore the dividing-by-$2$ and predecessor functions are discontinuous when the underlying topology is $\tau_{M}$.

\end{proof}

\begin{question}
 Is there a topological model of arithmetic $M$ having no continuous bijection between $M^2$ and $M$?
\end{question}

Note that subtraction and division are continuous with respect to the final-digits topology, in any model of arithmetic, by an analogue of the school-child's observation in the proof of the theorem~\ref{maintheorem}.

Similar ideas as in theorem~\ref{Theorem.Dividing-by-2-discontinuous} provide a negative answer to question~\ref{Question.If-one-then-all?}. Recall that a linearly ordered structure $M$ is \emph{$\kappa$-like}, if the cardinality of $M$ is $\kappa$, but the cardinality of every proper initial segment of $M$ is less than $\kappa$.  It is well-known that every countable model of $\PA$ has a $\kappa$-like elementary extension for every infinite cardinal $\kappa$.

\begin{theorem}\label{Theorem.Some-not-all-non-kappa-like}
 In every infinite cardinality $\kappa$, there is a topological space $X$ of that size continuously supporting exactly the models of arithmetic of size $\kappa$ that are not $\kappa$-like. Thus, $X$ continuously supports some, but not all models of arithmetic of this size.
\end{theorem}

\begin{proof}
Let $X$ be a space consisting of the indiscrete topology on a set of size $\kappa$, augmented with $\kappa$ many additional isolated points. If $M$ is a model of arithmetic of size $\kappa$, then $M$ is continuous with respect to the topology $\tau\restrict b$, where $\tau$ is the discrete topology and $b\in M$. If $M$ is not $\kappa$-like, then it has an element $b$ having $\kappa$ many predecessors, in which case $\tau\restrict b$ has $\kappa$ many isolated points and an indiscrete block of $\kappa$ many additional points, making the space homeomorphic to $X$. But if the model $M$ is $\kappa$-like, then it can have no continuous presentation on $X$, for in this case, the isolated points would have to be unbounded in $M$, and thus by lemma~\ref{Lemma.Upward-homogeneity} it would follow that every point would have to be isolated, contrary to the nature of $X$. So $X$ continuously supports exactly the non-$\kappa$-like models of arithmetic.
\end{proof}

We can prove a weak dual to this result.

\begin{theorem}\label{Theorem.kappa-like}
 For every infinite cardinal $\kappa$, there is a topological space $Y$ supporting a $\kappa$-like model, such that furthermore, only $\kappa$-like models are continuously presentable upon $Y$.
\end{theorem}

\begin{proof}
Let $M$ be any $\kappa$-like model, and let $Y$ be the initial-segment topology of $M$. Suppose that $N$ is continuously presented on this same topology. If $N$ is not $\kappa$-like, then it has a point $b$ with $\kappa$ many predecessors $a<b$ in $N$. Let $U$ be a neighborhood of $b$, which is not the whole space. It follows that $U$ has size less than $\kappa$, and therefore it omits $\kappa$ many $a<b$. For any such $a$, there is $k_a$ such that $a+k_a=b$, and all such $k_a$ are distinct. By continuity, there is an open neighborhood $W$ of $a$ such that $a'+k_a\in U$ for all $a'\in W$. Since the open sets of $X$ are linearly ordered, $a\in W$ and $a\notin U$, it follows that $U\of W$ and in particular, $b\in W$. Thus, $b+k_a\in U$ for all these $a$, which places $\kappa$ many elements into $U$, contrary to assumption.
\end{proof}

Thus, we have two topological spaces of size $\kappa$, namely the space $X$ of theorem~\ref{Theorem.Some-not-all-non-kappa-like} and the space $Y$ of theorem~\ref{Theorem.kappa-like}, each of which supports a topological model of arithmetic, but no model is continuously presented upon both of them. In this sense, the topology can make distinctions in the arithmetic structure. In the case $\kappa=\omega$, this provides a negative answer to question \ref{Question.Topology-cannot-distinguish-the-theory}, since there is only one $\omega$-like model of arithmetic: the standard model.

\begin{corollary}\label{Corollary.X-with-only-standard-model-presented}
 There is a countable topological space $X$, such that only the standard model of arithmetic $N$ is continuously presented upon $X$, and furthermore this model has a unique such presentation. In particular, this answers question \ref{Question.Topology-cannot-distinguish-the-theory} negatively.
\end{corollary}

\begin{proof}
Let $X$ be the initial-segment topology on the standard model $\N$, which is continuously presented upon $X$ by theorem \ref{Theorem.Initial/final-segment-topology}. The standard model $\N$ is $\omega$-like and it is up to isomorphism the only $\omega$-like model of arithmetic. Since theorem \ref{Theorem.kappa-like} shows that only $\omega$-like models have a continuous presentation upon $X$, it follows that only $\N$ admits such a presentation. In particular, all models presented upon $X$ have the same theory (true arithmetic), which provides a strong negative answer to question \ref{Question.Topology-cannot-distinguish-the-theory}.

Let us show that the presentation is unique. Suppose that $M=\<\N,\oplus,\otimes>$ is a model of arithmetic that is continuous with respect to the initial-segment topology on $\N$. Since $0^\N$ is the only isolated point, it follows from lemma \ref{Lemma.Upward-homogeneity} that $0^M=0^\N$. Since $a\oplus1^M=1^\N$ for some $a$ and $\set{0,1^\N}$ is open, there must be an open set $W$ such that $w\oplus 1^M\in\set{0,1^\N}$ for all $w\in W$. But $w\oplus 1^M=0$ is impossible, since $0^M=0^\N$, and so $w\oplus 1^M=1^\N$ for all $w\in W$. Therefore $W$ must have only one element, so $W=\set{0}$, the only open set with exactly one element. This shows that $a=0$ and therefore $1^M=1^\N$. Continuing similarly by induction, one sees that $M$ agrees completely with the standard model, and thus the presentation is unique.
\end{proof}

The arguments of theorem~\ref{Theorem.kappa-like} and corollary~\ref{Corollary.X-with-only-standard-model-presented} suggest that perhaps the initial segment topology can detect the order type of any model of arithmetic continuously presented upon it, and so we ask the following question.

\begin{question}
 If $M$ is a model of arithmetic, and $N$ is a model continuously presented with respect to the initial-segment topology of $M$, then must $M$ and $N$ have the same order type? Is this true for $\kappa$-like models $M$?
\end{question}

The final argument of corollary \ref{Corollary.X-with-only-standard-model-presented} shows the answer is affirmative when $\kappa=\omega$. More generally, the same argument shows that if $N$ is continuous with respect to the initial-segment topology of a model $M$, then the two models must have the same $0$, the same $1$ and so on through the standard part. The question is asking whether the pattern continues above this, so that the models have the same initial segments and hence the same order. Theorem~\ref{Theorem.kappa-like} shows that when $M$ is $\kappa$-like, then any $N$ presented on the initial-segment topology of $M$ will be $\kappa$-like, but it does not quite show for uncountable $\kappa$ that the order of $M$ and $N$ agree or that these orders are isomorphic. Note that models of arithmetic can have the same order without being isomorphic with respect to their addition and multiplication; for example, all countable nonstandard models of \PA\ have the same order type: $\omega+\Z\cdot\Q$.

\section{Topological models of arithmetic in the complex numbers}

In this last section, We provide further instances of topological models of arithmetic by revisiting the following result which presents uncountable models of arithmetic in the complex numbers. The result is due to Josef Ml{\v c}ek~\cite{Mlcek1973:A-representation-of-models-of-PA}; see also the second author's blog post~\cite{Hamkins.blog2018:Nonstandard-models-of-arithmetic-arise-in-the-complex-numbers}

\begin{theorem}\label{Theorem.Sub-ring-of-complex-numbers}
 Every model of  $\PA^-$ of size at most continuum has a presentation as a topological model, using a subspace of the real plane $\R^2$. Indeed, every such model is realized as a substructure of the complex numbers $\<\C,+,\cdot>$.
\end{theorem}

\begin{proof}
Suppose that $M$ is a model of $\PA^-$ of size at most continuum. Consider the integers of this model $\Z^M$, from which we may take the fraction field and then the algebraic closure, building a version of the field $A$ of algebraic numbers over $M$.\footnote{Of course, in a model of \PA, one could undertake a version of this construction inside $M$, but this would give rise to a different field, since one would be adding the roots of nonstandard degree polynomials this way. Both versions of the algebraic numbers would serve the purpose of this proof.} So $A$ is an algebraically closed field of characteristic zero, which has an elementary extension to such a field of size continuum. Since the theory of algebraically closed fields of characteristic zero is categorical in all uncountable powers, it follows that $A$ is isomorphic to a substructure of the complex numbers $\C$, and isomorphic to $\C$ itself if $M$ has size continuum. Since $M$ is isomorphic to a substructure of $\Z^M$, which sits inside $A$, it follows that $M$ is isomorphic to a substructure of $\C$, as claimed. Since complex number arithmetic is continuous with respect to the topology of the real plane, we thereby realize $M$ as a topological model of arithmetic using this subset of $\R^2$.
\end{proof}

In particular, every countable model of $\PA^-$ can be found as a substructure of the complex numbers. The same argument applied at higher cardinals shows that if $k$ is the uncountable algebraically closed field of characteristic zero, then every model of arithmetic $M\satisfies\PA^-$ of size at most the cardinality of $k$ embeds into $k$.

Note that by theorem \ref{Theorem.Polish-space-no-uncountable-chains}, if $M$ is a $\kappa$-like model of arithmetic for some uncountable $\kappa$ at most the continuum, then no subspace of $\R^2$ that continuously supports $M$ is analytic.

\small

\printbibliography

\end{document}

%% file: MathMacrosJDH.tex
\newtheorem{theorem}{Theorem}
\newtheorem{maintheorem}[theorem]{Main Theorem}

\newtheorem*{maintheorem*}{Main Theorem}
\newtheorem*{maintheorems*}{Main Theorems}
\newtheorem{corollary}[theorem]{Corollary}
\newtheorem*{corollary*}{Corollary}
\newtheorem*{corollaries*}{Corollaries}

\newtheorem{lemma}[theorem]{Lemma}

\newtheorem{question}[theorem]{Question}
\newtheorem*{question*}{Question}

\newtheorem*{questions*}{Questions}
\newtheorem{mainquestion}[theorem]{Main Question} 
\newtheorem*{mainquestion*}{Main Question} 
\newtheorem*{openquestion*}{Open Question} 

\newtheorem{remark}[theorem]{Remark}

\newcommand{\QED}{\end{proof}}

\def\proclaim[#1]{{\bf #1}}
\def\BF#1.{{\bf #1.}}

\def\says#1:#2\par{\item[#1] #2\par}

\newcommand{\Sierpinski}{Sierpi\'{n}ski}

\newcommand{\C}{{\mathbb C}}

\newcommand{\N}{{\mathbb N}}

\newcommand{\Q}{{\mathbb Q}}

\newcommand{\Z}{{\mathbb Z}}
\newcommand{\R}{{\mathbb R}}

\makeatletter
\newcommand{\dotminus}{\mathbin{\text{\@dotminus}}}
\newcommand{\@dotminus}{%
  \ooalign{\hidewidth\raise1ex\hbox{.}\hidewidth\cr$\m@th-$\cr}%
}
\makeatother

\newcommand{\of}{\subseteq}

\newcommand{\set}[1]{\{\,{#1}\,\}}

\newcommand{\restrict}{\upharpoonright} 
\newcommand{\satisfies}{\models}

\newcommand{\union}{\cup}

\newcommand{\intersect}{\cap}

\newcommand{\smalllt}{\mathrel{\mathchoice{\raise2pt\hbox{$\scriptstyle<$}}{\raise1pt\hbox{$\scriptstyle<$}}{\raise0pt\hbox{$\scriptscriptstyle<$}}{\scriptscriptstyle<}}}
\newcommand{\smallleq}{\mathrel{\mathchoice{\raise2pt\hbox{$\scriptstyle\leq$}}{\raise1pt\hbox{$\scriptstyle\leq$}}{\raise1pt\hbox{$\scriptscriptstyle\leq$}}{\scriptscriptstyle\leq}}}

   \def\DHLhksqrt#1#2{%
   \setbox0=\hbox{$#1\sqrt{#2\,}$}\dimen0=\ht0
   \advance\dimen0-0.2\ht0
   \setbox2=\hbox{\vrule height\ht0 depth -\dimen0}%
   {\box0\lower0.4pt\box2}}
\newcommand{\boolval}[1]{\mathopen{\lbrack\!\lbrack}\,#1\,\mathclose{\rbrack\!\rbrack}}
\def\[#1]{\boolval{#1}}
\newbox\gnBoxA
\newdimen\gnCornerHgt
\setbox\gnBoxA=\hbox{\tiny$\ulcorner$}
\global\gnCornerHgt=\ht\gnBoxA
\newdimen\gnArgHgt
\def\gcode #1{%
\setbox\gnBoxA=\hbox{$#1$}%
\gnArgHgt=\ht\gnBoxA%
\ifnum     \gnArgHgt<\gnCornerHgt \gnArgHgt=0pt%
\else \advance \gnArgHgt by -\gnCornerHgt%
\fi \raise\gnArgHgt\hbox{\tiny$\ulcorner$} \box\gnBoxA %
\raise\gnArgHgt\hbox{\tiny$\urcorner$}}
\newcommand{\UnderTilde}[1]{{\setbox1=\hbox{$#1$}\baselineskip=0pt\vtop{\hbox{$#1$}\hbox to\wd1{\hfil$\sim$\hfil}}}{}}
\newcommand{\Undertilde}[1]{{\setbox1=\hbox{$#1$}\baselineskip=0pt\vtop{\hbox{$#1$}\hbox to\wd1{\hfil$\scriptstyle\sim$\hfil}}}{}}
\newcommand{\undertilde}[1]{{\setbox1=\hbox{$#1$}\baselineskip=0pt\vtop{\hbox{$#1$}\hbox to\wd1{\hfil$\scriptscriptstyle\sim$\hfil}}}{}}
\newcommand{\UnderdTilde}[1]{{\setbox1=\hbox{$#1$}\baselineskip=0pt\vtop{\hbox{$#1$}\hbox to\wd1{\hfil$\approx$\hfil}}}{}}
\newcommand{\Underdtilde}[1]{{\setbox1=\hbox{$#1$}\baselineskip=0pt\vtop{\hbox{$#1$}\hbox to\wd1{\hfil\scriptsize$\approx$\hfil}}}{}}

\def\<#1>{\left\langle#1\right\rangle}

\newcommand{\PA}{{\rm PA}}

\newcommand{\cell}[1]{\boxit{\hbox to 17pt{\strut\hfil$#1$\hfil}}}
\newcommand{\head}[2]{\lower2pt\vbox{\hbox{\strut\footnotesize\it\hskip3pt#2}\boxit{\cell#1}}}
\newcommand{\boxit}[1]{\setbox4=\hbox{\kern2pt#1\kern2pt}\hbox{\vrule\vbox{\hrule\kern2pt\box4\kern2pt\hrule}\vrule}}
\newcommand{\Col}[3]{\hbox{\vbox{\baselineskip=0pt\parskip=0pt\cell#1\cell#2\cell#3}}}
\newcommand{\tapenames}{\raise 5pt\vbox to .7in{\hbox to .8in{\it\hfill input: \strut}\vfill\hbox to
.8in{\it\hfill scratch: \strut}\vfill\hbox to .8in{\it\hfill output: \strut}}}
\newcommand{\Head}[4]{\lower2pt\vbox{\hbox to25pt{\strut\footnotesize\it\hfill#4\hfill}\boxit{\Col#1#2#3}}}
\newcommand{\Dots}{\raise 5pt\vbox to .7in{\hbox{\ $\cdots$\strut}\vfill\hbox{\ $\cdots$\strut}\vfill\hbox{\
$\cdots$\strut}}}